\documentclass{amsart}
\usepackage[utf8]{inputenc}

\usepackage[T1]{fontenc}
\usepackage{amsmath,amssymb,amsthm}
\usepackage{graphicx}
\usepackage{verbatim}
\usepackage[utf8]{inputenc}
\usepackage{lmodern}
\usepackage{tikz-cd}
\usepackage{hyperref}
\usepackage{graphicx}
\usepackage[english]{babel}
\usepackage{color}
\usepackage[margin=1in]{geometry}
\usepackage{cite}
\usepackage{mathtools}

\newtheorem{theorem}{Theorem}[section]
\newtheorem*{theorem*}{Theorem}
\newtheorem{definition}[theorem]{Definition}

\newtheorem{lemma}[theorem]{Lemma}
\newtheorem{corollary}[theorem]{Corollary}
\newtheorem{remark}[theorem]{Remark}

\newtheorem{proposition}[theorem]{Proposition}

\newcommand{\Hom}{\mathsf{Hom}}
\newcommand{\RR}{\mathbb{R}}
\newcommand{\CC}{\mathbb{C}}
\newcommand{\ZZ}{\mathbb{Z}}
\newcommand{\ZZtwo}{\ZZ/2\ZZ}

\newcommand{\GL}{\mathsf{GL}}
\newcommand{\SL}{\mathsf{SL}}
\newcommand{\SO}{\mathsf{SO}}
\newcommand{\SU}{\mathsf{SU}}

\newcommand{\U}{\mathsf{U}}

\newcommand{\End}{\mathsf{End}}

\renewcommand{\O}{\mathsf{O}}

\newcommand{\CP}{\mathbb{CP}}
\newcommand{\F}{\mathcal{F}}
\newcommand{\brak}[1]{\left<#1\right>}

\renewcommand{\P}{\mathcal{P}}
\newcommand{\floor}[1]{\lfloor #1 \rfloor}

\newcommand{\Sc}{\mathcal{S}}
\newcommand{\abs}[1]{\left|#1\right|}
\newcommand{\spec}{\mathsf{Spec}}
\newcommand{\NN}{\mathbb{N}}
\newcommand{\QQ}{\mathbb{Q}}

\newcommand{\bE}{\mathbf{E}}
\newcommand{\bI}{\mathbf{I}}

\newcommand{\St}{\mathsf{St}}

\newcommand{\define}[1]{\textbf{#1}}

\title{On the Existence of Parseval Frames for Vector Bundles}
\author{Samuel A.\ Ballas}
\address{Department of Mathematics, Florida State University, Tallahassee, FL 32306}
\email{sballas@fsu.edu}

\author{Tom Needham}
\address{Department of Mathematics, Florida State University, Tallahassee, FL 32306}
\email{tneedham@fsu.edu}

\author{Clayton Shonkwiler}
\address{Department of Mathematics, Colorado State University, Fort Collins, CO 80523}
\email{clayton.shonkwiler@colostate.edu}
\date{\today}

\begin{document}

\begin{abstract}
    Frames in finite-dimensional vector spaces are spanning sets of vectors which provide redundant representations of signals. The \textit{Parseval frames} are particularly useful and important, since they provide a simple reconstruction scheme and are maximally robust against certain types of noise. In this paper we describe a theory of frames on arbitrary vector bundles---this is the natural setting for signals which are realized as parameterized families of vectors rather than as single vectors---and discuss the existence of Parseval frames in this setting. Our approach is phrased in the language of $G$-bundles, which allows us to use many tools from classical algebraic topology. In particular, we show that orientable vector bundles always admit Parseval frames of sufficiently large size and provide an upper bound on the necessary size. We also give sufficient conditions for the existence of Parseval frames of smaller size for tangent bundles of several families of manifolds, and provide some numerical evidence that Parseval frames on vector bundles share the desirable reconstruction properties of classical Parseval frames. 
\end{abstract}

\maketitle

\section{Introduction}

A \emph{frame} for a (real or complex) Hilbert space $V$ is a sequence of vectors $(f_j)_{j \in \mathcal{J}}$ for which there are real constants $a,b > 0$ satisfying the bounds
\begin{equation}\label{eqn:frame_bounds}
a \|v\|^2 \leq \sum_{j \in \mathcal{J}} |\langle v, f_j \rangle |^2 \leq b \|v\|^2
\end{equation}
for all $v \in V$. A frame is called \emph{tight} if one can take $a = b$ and \emph{Parseval} (or \emph{normalized}) if one can take $a = b  = 1$. An orthonormal basis for $V$ is a Parseval frame, due to \emph{Parseval's identity} (the source of the nomenclature), but one generally considers frames which are \emph{overcomplete} or \emph{redundant}, in the sense that the collection $\{f_j\}_{j \in \mathcal{J}}$ is linearly dependent (although a frame must always be spanning, due to the lower bound in~\eqref{eqn:frame_bounds}). Frames appear in signal processing applications, where the Hilbert space represents a space of signals and the sequence of inner products $(\langle v, f_j \rangle)_{j \in \mathcal{J}}$ represents an encoding of the signal $v$ with respect to the frame. This encoding scheme is preferred in many applications over encoding with respect to an orthonormal basis due to its robustness properties: overcomplete Parseval frames are more resilient to both additive measurement noise (e.g.,~\cite{munch1992noise,benedetto2001wavelet}) and information loss due to transmission errors (e.g.,~\cite{bodmann2005frames,casazza2003equal}), where resilience is quantified via bounds on the error incurred when reconstructing the signal from its corrupted measurements.

Due to computational concerns, one typically deals in practice with a finite-dimensional approximation of the signal space. As a result, frame theory in the setting of finite-dimensional Hilbert spaces has been intensely studied in the past few decades---see the survey monographs~\cite{kovavcevic2008introduction,waldron2018introduction} and the references therein. In the finite setting, one typically considers a frame to be a \emph{finite} sequence of vectors, and the existence of constants satisfying~\eqref{eqn:frame_bounds} is equivalent to the statement that the sequence gives a spanning set of $V$. Due to this simplified interpretation of the basic definition, there is particular emphasis on studying finite frames with more interesting prescribed properties, such as that of being Parseval. Much is known about the collection of all Parseval frames consisting of a fixed number of vectors in a fixed finite-dimensional Hilbert space; for example, it is non-empty and is  equivalent to a well-known algebraic variety (a \emph{Stiefel manifold})~\cite{strawn2011finite}.

In this paper, we study existence properties of smoothly-varying families of frames. That is, we consider a rank-$k$ (real or complex) vector bundle $E \to M$ over a smooth manifold $M$ equipped with a smoothly varying family of (Hermitian, in the complex case) inner products. In this context, given an integer $n \geq k$, we say that a collection $(\sigma_j)_{j=1}^n$ of sections $\sigma_j:M\to E$ is a \emph{(Parseval) frame of size $n$} for $E\to M$ if for each $p\in M$, $(\sigma_j(p))_{j=1}^n$ is a (Parseval) frame for the fiber $E_p$. In this paper we address the following.

\smallskip
\noindent{\bf Main Question:} What conditions on the bundle $E \to M$ guarantee the existence of Parseval frame of size $n$ on $E$?

\smallskip
\noindent This question has already been considered in~\cite{freeman2012moving,freeman2012moving2,kuchment2008tight} (prior work is discussed more thoroughly below in Section~\ref{subsec:previous_work}). This paper takes a somewhat more general perspective than these earlier works, and was inspired by the observation that the Main Question should be susceptible to methods from classical algebraic topology, such as obstruction theory~\cite{Steenrod_Fibre_Bundles}. We make partial progress on the Main Question from several fronts, as we outline below in Section~\ref{subsec:main_results}.

While the primary motivation for this work is theoretical, we also believe that there are applications-oriented reasons for one to consider frames on vector bundles. On the theoretical side, we find it natural to explore the extent to which results from the classical theory of frames for a single finite-dimensional vector space can be generalized to frames on vector bundles, or collections of frames on parametric families of vector spaces. Our results show that the topic is likely to be rather subtle, and offers many directions for future work. On the practical side, it is not hard to imagine situations in which the signal of interest is not just a single vector in a fixed vector space, but rather a collection of vectors which are smoothly parameterized by some manifold---that is, the signal is a smooth section of a vector bundle. For example, such smooth sections are useful in geometric deep learning as a way to give local coordinates for feature spaces sitting over each point in an underlying data manifold~\cite{melzi2019gframes,bronstein2021geometric}. A frame for the vector bundle provides a method to encode these smooth sections. Recovery of a signal in this setting involves inverting a certain linear operator (the \emph{frame operator}, see Section~\ref{sec:frames_in_vector_spaces} for details) for the frame at each point of the manifold, a task which is both computationally taxing and potentially numerically unstable. This computational burden is removed under the condition that the frame is Parseval at each point, in which case the reconstruction process is much more straightforward. Beyond this basic computational benefit, we expect that Parseval frames on vector bundles will enjoy enhanced robustness to noise over general frames, in analogy with the results from classical frame theory. We provide some initial evidence for this conjecture in a numerical experiment; see Section~\ref{subsec:numerical_experiment}.
A rigorous study of these practical considerations is beyond the scope of this paper, but is planned to be the subject of future work.

\subsection{Previous work}\label{subsec:previous_work} Parseval frames on vector bundles have been studied previously in the literature; we now briefly survey some prior results. To our knowledge, the earliest results on the topic appear in~\cite{kuchment2008tight} in the context of mathematical physics. The interest is in the existence of a Parseval frame for a certain analytic vector bundle over a torus  and an upper bound on the required number of frame vectors is derived. This work is continued in~\cite{kuchment2016overview,auckly2018parseval}, where the bound on the number of required vectors is improved. These results are described in more detail below in Section~\ref{subsec:main_results}.  Parseval frames on arbitrary vector bundles are considered in~\cite{freeman2012moving}. Here, analogues of the Naimark--Han--Larson Dilation Theorem~\cite{han2000frames}---that any Parseval frame can be realized as the image of an orthonormal basis for some higher-dimensional space under an orthogonal projection---are established in the vector bundle setting. In~\cite{freeman2012moving2}, the question of the existence of \emph{unit norm tight frames} (that is, tight frames whose frame vectors are unit vectors) on $n$-spheres is considered. There it is shown that every odd-dimensional sphere admits a unit norm tight frame, in contrast to the well-known topological result that the spheres which admit global orthonormal frames are only those of dimension $1$, $3$, or $7$. However, unit norm tight frames are still rather uncommon since the sections comprising a unit norm tight frame are each nowhere-vanishing. Thus, these types of frames can only exist for vector bundles with trivial Euler class. 

There have been several other papers which illustrate or exploit connections between geometry/topology and frame theory. These works have some overlap in terminology, but generally deal with quite different aspects of the subject than this paper. For example,~\cite{agrawal2015fiber} considers bundle-theoretic aspects of frame theory, but considers bundle structures on spaces of frames, rather than frames defined for bundles. The recent paper~\cite{bownik2022parseval} treats Parseval frames for Riemannian manifolds, but the existence questions considered in the present paper are not pertinent there. There has been much work on the geometry and topology of moduli spaces of Parseval frames on a fixed Hilbert space---see, for example,~\cite{dykema2003manifold,
strawn2011finite,cahill2017connectivity} or prior work of the last two authors of this paper~\cite{needham2023fusion,needham2022admissibility,needham2022toric,needham2021symplectic,mixon2021three}. None of these papers consider frames on vector bundles, and the extent to which their results generalize will be the subject of future work.

\subsection{Main results}\label{subsec:main_results} We now describe our main contributions. Our first main results concern the existence of Parseval frames on both real and complex vector bundles. Roughly speaking, these results show that for a vector bundle $E\to M$ (satisfying mild hypotheses) equipped with a smoothly varying fiberwise inner product (resp.\ Hermitian form), which we refer to as a \emph{Riemannian structure} (resp.\ \emph{Hermitian structure}), there is a natural number $n$ so that $E\to M$ admits Parseval frames of size $n$. Moreover, the value $n$ is explicit, depends only on the rank $k$ of the vector bundle and the dimension $d$ of the base manifold, and grows linearly with respect to both $k$ and $d$. 


\begin{theorem*}[Theorem \ref{thm:parseval_sections_exist}]
Let $E\to M$ be a rank-$k$ orientable real vector bundle over a $d$-dimensional smooth manifold. Suppose further that $E\to M$ is equipped with a Riemannian structure. Then there exists a Parseval frame of size $d+k$ for $E\to M$. 
\end{theorem*}

\noindent There is also a version of this result for complex vector bundles. In this case we find the bounds are even lower.

\begin{theorem*}[Theorem \ref{thm:parseval_sections_exist_complex}]
Let $E\to M$ be a rank-$k$ complex vector bundle over a $d$-dimensional smooth (real) manifold. Suppose further that $E\to M$ is equipped with a Hermitian structure. Then there exists a Parseval frame of size $\floor{d/2}+k$ for $E\to M$. 
\end{theorem*}

\noindent Note that no orientability hypothesis is needed here since complex vector bundles are automatically orientable. We also show that, while one might expect that the conditions for existence of a Parseval frame on a (real or complex) vector bundle are more restrictive than those for existence of a general frame, this is not the case.

\begin{theorem*}[Theorem \ref{thm:parseval_deformation_retract} and Corollary \ref{cor:parseval_if_and_only_if_frame}]
Let $E \to M$ be an orientable real vector bundle endowed with a Riemannian structure. There is a size-$n$ Parseval frame for $E \to M$ if and only if there is a size-$n$ frame.
\end{theorem*}

\noindent The key idea here is that the space of all frames of a fixed size on a Hilbert space deformation retracts onto the space of all Parseval frames of that size. Moreover, the retraction preserves the underlying fiber bundle structures.

As mentioned in the previous section, in~\cite{auckly2018parseval, kuchment2008tight,kuchment2016overview} the authors provide bounds for the dimensions in which Parseval frames exist for certain rank-$k$ complex vector bundles (so called \emph{Bloch bundles}) $E\to T^d$ over a $d$-dimensional torus. In this context, it is shown in~\cite{kuchment2008tight} that there exist a Parseval frame for $E\to T^d$ of size at most $2^dk$. This bound grows exponentially in $d$, and so one sees that the bound in Theorem~\ref{thm:parseval_sections_exist_complex} is considerably sharper. Subsequently, in~\cite{auckly2018parseval}, the authors are able to improve the above bound in the cases where $d=3$ and $d=4$. Specifically, they are able to show that when $d=3$ and $E\to T^3$ is a rank-$k$ Bloch bundle there exists a Parseval frame of size $k+1$, and when $d=4$ and $E\to T^4$ is a rank-$k$ Bloch bundle there exists a Parseval frame of size $k+2$. Note that in both of these cases, the results agree with the bounds coming from Theorem \ref{thm:parseval_sections_exist_complex} (i.e. $\floor{3/2}=1$ and $\floor{4/2}=2)$. 

Once it is known that Parseval frames exist for arbitrary (orientable) vector bundles, it is natural to ask how small such a frame can be. If $E\to M$ is a rank-$k$ vector bundle then the smallest possible size of a Parseval frame is $k$. Furthermore, it is not difficult to see that a Parseval frame of size $k$ provides a global trivialization of $E$, and so such a Parseval frame can exist if and only if $E$ is a trivial bundle. Thus, given a non-trivial rank-$k$ vector bundle, the best one can hope for is a Parseval frame of size $k+1$. With this in mind, we restrict our attention to tangent bundles over $d$-manifolds and give the following characterization of the existence of a size-$(d+1)$ Parseval frame, in terms of \emph{stable parallelizability}, a well-studied concept in classical geometric and algebraic topology. 

\begin{theorem*}[Theorem \ref{thm:stably_parallel}]
If $M$ is a stably parallelizable orientable Riemannian $d$-manifold, then $TM$ admits a Parseval frame of size $d+1$. Moreover, if $H^1(M,\ZZtwo)=0$ then the converse is also true.
\end{theorem*}

This result immediately yields many examples of $d$-manifolds whose tangent bundle admits size-$(d+1)$ Parseval frames: orientable surfaces, 4-manifolds with additional restrictions on their characteristic classes, certain lens spaces (of arbitrarily high dimension), and homology spheres. On the other hand, we use Theorem \ref{thm:stably_parallel} to give  examples of $d$-manifolds which \emph{do not} admit Parseval frames of size $d+1$ for all $d > 3$ (Theorem~\ref{thm:no_d+1_frames}).

A frame $(f_j)_{j=1}^n$ in a real (resp.\ complex) Hilbert space $V$ can alternatively be thought of as a surjective linear map from $\RR^n$ (resp.\ $\CC^n$) to $V$. To such a map, one can associate its singular values, which form a collection of positive real numbers. The condition that the frame is Parseval is equivalent to all singular values being equal to 1. At the other extreme is the case where the singular values of the frame are pairwise distinct. We call such frames \emph{generic} (see Section~\ref{sec:spectra} for a precise definition)---this is warranted, since an open dense subset of frames have this property. For a vector bundle $E\to M$, we correspondingly say that a frame is \emph{generic} if it is generic at each point of $M$, although this turns out to be somewhat of a misnomer. That is, while one might naively expect that most collections of sections will form a generic frame, we show that the topology of $M$ and $E$ can provide obstructions to finding such generic frames.

\begin{theorem*}[Theorem \ref{thm:obstruction_to_generic_sections}]
Let $E\to M$ be an oriented rank-$k$ vector bundle over a connected $d$-manifold, endowed with a smoothly-varying family of fiberwise inner products. Suppose that $E\to M$ has non-trivial Euler class and that $H_1(M,\ZZtwo)=0$. Then $E$ does not admit sections of generic frames.
\end{theorem*}

In fact, the full statement of Theorem~\ref{thm:obstruction_to_generic_sections} demonstrates the non-existence of sections with a weaker notion of genericity. In particular, we are able to show that simply connected 4-manifolds \emph{never} admit generic frames, in this weak sense (Corollary~\ref{cor:simply_connected_4_manifolds}).

\subsection{Outline of the paper}\label{sec:outline} The main technique of this paper is to reformulate the existence of Parseval frames of size $n$ for $E\to M$ as a question about the existence of sections of a certain bundle $\P^n(E)\to M$ that we call the \emph{$n$-Parseval bundle of $E$}. Once this is accomplished, techniques from classical obstruction theory can be used to ascertain when such sections exist. More precisely, sections of $\P^n(E)$ are obstructed by certain cohomology classes on $M$ with values in the homotopy groups of the fibers of $\P^n(E)$. In Section~\ref{sec:frames_in_vector_spaces} we describe the space of all Parseval frames of fixed size in a (finite-dimensional) Hilbert space. Here we observe the well known fact that the space of Parseval frames is diffeomorphic to a \emph{Stiefel manifold} (see, e.g.,~\cite{strawn2011finite}). This is a useful perspective to take with a view toward generalizing to vector bundles; indeed, this Stiefel manifold will end up being the fiber of the Parseval bundle mentioned above. Its low-dimensional homotopy groups are well known, and this knowledge is crucial for us as we attempt to analyze the obstructions mentioned above. In Section~\ref{sec:frames_on_vector_bundles} we describe the construction of the Parseval bundle in detail. In Section~\ref{sec:existence_of_frames} we prove our general existence theorems by showing that the relevant obstructions to finding sections all vanish, and then study the tangent bundle case in more detail through the lens of stable parallelizability. Finally, in Section~\ref{sec:spectra} we define the notion of \emph{spectral covers} and use them to show that in certain situations it is not possible to find generic frames. 

\subsection*{Acknowledgements} The authors would like to thank Eric Klassen for suggesting the proof of Proposition~\ref{prop:Parseval_hom_eq_to_frames} and Ettore Aldrovandi for several useful discussions regarding Whitehead towers and obstructions. This research was partially supported by the National Science Foundation (DMS--2107808, Tom Needham; DMS--2107700, Clayton Shonkwiler). 

\section{Frames in vector spaces}\label{sec:frames_in_vector_spaces}

Before discussing frames on vector bundles, we begin with a treatment of frames in vector spaces. In anticipation of migrating these spaces to the setting of vector bundles, we focus on describing spaces of frames as homogeneous spaces of various Lie groups.

\subsection{The space of frames}\label{subsec:real_case}
Let $V$ be a $k$-dimensional real vector space. 

\begin{definition}[Frame]
    For $n \geq k$, an \define{$n$-frame in $V$} is a surjective linear map $A\in \Hom(\RR^n,V)$. The space of all $n$-frames in $V$ will be denoted $\F^n(V)$. 
\end{definition}

\begin{remark}\label{rmk:n_frame_definition}
   The image of the standard basis of $\RR^n$ under a frame $A$ then provides a spanning set for $V$. Moreover, any spanning sequence of $n$ vectors can be realized in this manner, so we see that this definition of a frame is equivalent to the one given in the introduction. We also observe there that if $n=k$, then a frame is equivalent to a choice of basis for $V$.
\end{remark}
 
Let us describe the basic topological structure of $\F^n(V)$. The standard topology on $\Hom(\RR^n,V)$ is homeomorphic to the Euclidean topology on the vector space of $k\times n$ matrices; $\F^n(V)$ is a dense open subset of $\Hom(\RR^n,V)$, and we endow it with the induced smooth structure. Moreover, $\F^n(V)$ is diffeomorphic to a familiar  space: let $\St_k^{\mathrm{nc}}(\RR^n)$ denote the \define{non-compact Stiefel manifold}, consisting of all sets of $k$ linearly independent vectors in $\RR^n$. 

\begin{proposition}\label{prop:non-compact_stiefel}
    The space of $n$-frames $\F^n(V)$ is diffeomorphic to the non-compact Stiefel manifold $\St_k^{\mathrm{nc}}(\RR^n)$. 
\end{proposition}

\begin{proof}
    Fix an inner product $\langle \cdot, \cdot \rangle$ on $V$. For $A \in \F^n(V)$, let $A^T \in \Hom(V,\RR^n)$ be the transpose of $A$ with respect to $\langle \cdot, \cdot \rangle$ and the standard inner product on $\RR^n$. Since $\mathrm{rank}(A) = \mathrm{rank}(A^T)$, the map $A^T$ is injective. Fixing a choice of basis for $V$, the image of this basis under $A^T$ gives an element of $\St_k^{\mathrm{nc}}(\RR^n)$, and any element can be realized in this way. We can therefore identify elements of $\St_k^{\mathrm{nc}}(\RR^n)$ with linear maps $A^T$, and the map $\F^n(V) \to \St_k^{\mathrm{nc}}(\RR^n)$ defined by $A \mapsto A^T$ gives the desired diffeomorphism.
\end{proof}

Let $\GL(W)$ denote the Lie group of invertible linear transformations of a vector space $W$. When ${W = \RR^n}$, we write $\GL(n) = \GL(\RR^n)$, and we identify $\GL(k) \cong \GL(V)$ for the $k$-dimensional vector space $V$ by picking bases. We now realize $\F^n(V)$ as a homogeneous space in a somewhat non-standard way. The group ${\hat G \coloneqq \GL(k)\times \GL(n)}$ acts on $\F^n(V)$ via $(M,N)\cdot A=MAN^{-1}$. The existence of the Smith normal form shows that $\hat G$ acts transitively on $\F^n(V)$ and gives a preferred representative of the orbit as the linear map which can be written with respect to some choice of bases as the block matrix
\begin{equation}\label{eqn:E_matrix}
\bE=\begin{pmatrix}
    \bI_{k}& \mathbf{0}_{k,n-k}
\end{pmatrix},
\end{equation}
where $\bI_k$ and $\mathbf{0}_{k,n-k}$ are identity and zero matrices of the corresponding sizes, respectively. The stabilizer of the linear map associated to $\bE$ is  $\hat H = \GL(k) \times \GL(n-k)$. We have proved the following.

\begin{proposition}\label{prop:frames_homogeneous_space}
    With the notation above, $\F^n(V) \cong \hat G/\hat H = (\GL(k) \times \GL(n))/(\GL(k) \times \GL(n-k))$.
\end{proposition}

\begin{remark}\label{rem:alternative_homogeneous_structure}
    It is well-known that $\St_k^{\mathrm{nc}}(\RR^n)$ can be realized as the homogeneous space $\GL(n)/\GL(n-k)$, and Proposition \ref{prop:non-compact_stiefel} implies that $\F^n(V) \cong \GL(n)/\GL(n-k)$ as well.
    However, it will be useful for us to consider the alternative homogeneous space structure of Proposition \ref{prop:frames_homogeneous_space}, which is apparently somewhat redundant. This perspective will prove to be convenient when we transition to the vector bundle setting. The basic idea is that the group $\hat G$ is the structure group of a vector bundle with fibers of the form $\Hom(\RR^n,V)$, so the realization of $\F^n(V)$ via the action of $\hat G$ becomes relevant when constructing a certain associated bundle; see Section \ref{sec:parseval_bundle} below for details.
\end{remark}

\subsection{The space of Parseval frames.}
Next, suppose that $V$ has been endowed with an inner product. We will abuse notation and use $\brak{\cdot,\cdot}$ for both the inner product on $V$ and the standard inner product on $\RR^n$. By taking the composition of $A$ and its transpose, we define the \define{frame operator} of $A$, $AA^T\in \End(V)$, where $\End(V)$ is the space of linear maps from $V$ to itself. If $A$ is a frame then the frame operator is invertible, so $AA^T\in \GL(V)$. For reasons described in the introduction, a case of particular interest is when $AA^T=I_V$, the identity operator on $V$. 

\begin{definition}[Parseval frames]
    A frame $A\in \F^n(V)$ with frame operator $AA^T$ equal to the identity is called a \define{Parseval frame}. More generally, if $AA^T=cI_V$ for some $c>0$ then we say that $A$ is a \define{tight frame}. We will will denote the space of $n$-Parseval frames as $\mathcal{P}^n(V)$. 
\end{definition}

Similar to Proposition \ref{prop:non-compact_stiefel}, the space of Parseval frames can be identified with a familiar space: let $\St_k(\RR^n)$ denote the \define{(compact) Stiefel manifold}, consisting of all sets of $k$ pairwise orthogonal unit vectors in~$\RR^n$. The following fact is familiar to the frame theory community (see, e.g.,~\cite{strawn2011finite,needham2021symplectic}), and is proved by the same method as Proposition \ref{prop:non-compact_stiefel}; that is, via the transpose map.

\begin{proposition}\label{prop:Parseval_same_as_Stiefel}
    The space of Parseval frames $\P^n(V)$ is diffeomorphic to the Stiefel manifold $\St_k(\RR^n)$.
\end{proposition}

\begin{remark}\label{rmk:parseval_orthonormal_basis}
    A simple observation is that if $n = k$, then Parseval frames are equivalent to orthonormal bases, since $\St_k(\RR^k)$ is, by definition, the space of orthonormal bases for $\RR^k \cong V$. 
\end{remark}

Let $\iota: \P^n(V)\hookrightarrow \F^n(V)$ denote the inclusion map.

\begin{proposition}\label{prop:Parseval_hom_eq_to_frames}
    The inclusion $\iota:\P^n(V)\hookrightarrow \F^n(V)$ is a homotopy equivalence
\end{proposition}

\begin{proof}
In light of Propositions \ref{prop:non-compact_stiefel} and \ref{prop:Parseval_same_as_Stiefel}, this is equivalent to the well-known fact that the inclusion $\St_k(\RR^n) \hookrightarrow \St_k^{\mathrm{nc}}(\RR^n)$ is a homotopy equivalence. In particular, this can be proved using the Gram--Schmidt orthonormalization algorithm. For the convenience of the reader, we give details of the proof, working directly in the spaces of frames.

The idea behind the proof is an application of the polar decomposition from linear algebra, but we include the details for completeness. Let $A\in \F^n(V)$, so that the frame operator $AA^T$ is symmetric and positive-definite.

Let $P$ be the unique positive-definite symmetric square root of $AA^T$, and let $B=P^{-1}A\in \F^n(V)$. Computing, we find that 
$$BB^T=P^{-1}AA^T(P^{-1})^T=P^{-1}AA^TP^{-1}=I,$$
where the last equality follows since $AA^T=P^2$. Thus $B\in \P^n(V)$ and we can write $A=PB$, so each frame can be decomposed as the product of a positive-definite symmetric operator and a Parseval frame. Furthermore, if $A=P'B'$ is a different such decomposition, then $AA^T=P'B'(B')^TP'=(P')^2$, and so $P'$ is a positive-definite square root of $AA^T$. By uniqueness of such square roots we see that $P=P'$ and hence $B=B'$ and this decomposition is unique. 

Next, define $f:\F^n(V)\to \P^n(V)$ by $f(A)=f(PB)=B$. It is easy to see that $f$ is a left inverse of $\iota$. Furthermore, the space of positive-definite symmetric matrices is a convex subset of the vector space $\End(V)$, so taking the straight-line homotopy to the identity endomorphism exhibits a homotopy from $\iota\circ f$ to the identity map on $\F^n(V)$. Thus $f$ is a homotopy inverse of $\iota$. 
\end{proof}

\subsection{Geometric structure of \texorpdfstring{$\P^n(V)$}{Parseval space}} Next, we provide an alternative and more geometric description of Parseval frames. Let $A \in \F^n(V)$ and consider the map $A^T:V \to \RR^n$. For all $v,w \in V$, we have
$
\brak{A^T v, A^T w} = \brak{AA^T v, w},
$
and it follows easily that $A$ is Parseval if and only if $A^T$ is an isometric embedding. Slightly more generally, we have the following proposition (which is well known---see, e.g.,~\cite[Section~2.1]{waldron2018introduction}).

\begin{proposition}\label{prop:parseval_tight_frame_condition}
A frame $A\in \F^n(V)$ is tight if and only if $A^T:V\to \RR^n$ is an isometric embedding followed by a homothety. Moreover, the frame is Parseval if and only if the homothety is trivial.  
\end{proposition}

As an application of Proposition~\ref{prop:parseval_tight_frame_condition}, we will show that Parseval frames can be realized as a homogeneous space. Let $\O(W)$ denote the Lie group of orthogonal linear transformations of an inner product space $W$. If $W = \RR^n$ with its standard inner product, we write $\O(n) = \O(\RR^n)$. Moreover, for our $k$-dimensional inner product space $V$, we identify $\O(k) \cong \O(V)$ by choosing bases. In light of Proposition \ref{prop:parseval_tight_frame_condition}, the group $G\coloneqq\O(k) \times \O(n)$ acts  on $\P^n(V)$ by $(M,N)\cdot A=MAN^{-1}=MAN^T$ (i.e., the restriction of the action of $\hat G$ on $\F^n(V)$), and it is straightforward to show that this action is transitive with stabilizer subgroup isomorphic to $H\coloneqq\O(k)\times \O(n-k)$. This proves the following.

\begin{proposition}\label{prop:parseval_homogeneous}
    With the notation above, $\P^n(V) \cong G/H = (\O(k) \times \O(n))/(\O(k) \times \O(n-k))$.
\end{proposition}

\begin{remark}
    This remark parallels Remark \ref{rem:alternative_homogeneous_structure}. It is well known that the Stiefel manifold $\St_k(\RR^n)$ is diffeomorphic to the homogeneous space ${\O(n)/\O(n-k)}$, where $\O(n)$ denotes the Lie group of orthogonal linear transformations. In light of the above discussion, we have 
    $
    \P^n(V) \cong \O(n)/\O(n-k)
    $
    as well. Similar to the situation with $\F^n(V)$, the homogeneous space structure described in Proposition \ref{prop:parseval_homogeneous} gives a useful perspective when constructing an associated bundle with fiber $\P^n(V)$.
\end{remark}

\subsection{Complex frame spaces.}\label{subsec:complex_case}

Much of the above discussion applies equally well to the case where $V$ is a finite-dimensional complex vector space. In this section we summarize the analogous results in this setting. Suppose that $V$ has complex dimension $k$. Then we define an \define{$n$-frame on $V$} to be a surjective element of $\Hom(\CC^n,V)$ and we continue to denote the space of $n$-frames by $\F^n(V)$. If we let $\hat G=\GL(\CC^k)\times \GL(\CC^n)$, then we find that once again $\hat G$ acts transitively on $\F^n(V)$ and that $\F^n(V)\cong \hat G/\hat H$, where $\hat H = \GL(\CC^k)\times \GL(\CC^{n-k})$ is the subgroup of $\hat G$ that stabilizes $\bE$, as defined in \eqref{eqn:E_matrix}. 

Next, suppose that $V$ is equipped with a positive-definite Hermitian form, which we again denote by $\langle \cdot,\cdot \rangle$. If we equip $\CC^n$ with the standard Hermitian form, then for $A\in \Hom(\CC^n,V)$ we denote the conjugate transpose with respect to these Hermitian forms as $A^\ast\in \Hom(V,\CC^n)$. As before, we define the \define{frame operator} of $A$ to be $AA^\ast\in \End(V)$ and we say that a frame is \define{Parseval} if the frame operator is the identity and \define{tight} if its frame operator is a scalar multiple of the identity. We again let $\P^n(V)$ denote the space of all Parseval frames of size $n$ on $V$.

As before, the map $A\mapsto A^\ast$ gives a diffeomorphism between $\P^n(V)$ and the space of linear maps ${V \to \CC^n}$ which preserve the corresponding Hermitian forms. Furthermore, if we let $\St_k(\CC^n)$ be the space of orthonormal subsets of $k$ unit vectors in $\CC^n$, then there is a diffeomorphism between $\P^n(V)$ and $\St_k(\CC^n)$. As a homogeneous space, we have $\St_k(\CC^n)\cong \U(n)/\U(n-k)$, where $\U(n)$ is the group of unitary operators on $\CC^n$. However, we instead consider $\P^n(V)$ as the homogeneous space $(\U(k) \times \U(n))/(\U(k) \times \U(n-k))$.

\section{Frames on vector bundles}\label{sec:frames_on_vector_bundles}

We will now generalize the notion of frames and Parseval frames from vector spaces to vector bundles. Throughout this section we let $M$ be a $d$-dimensional manifold and $E\to M$ a rank-$k$ orientable real vector bundle over $M$ (the complex case is treated in Section \ref{sec:complex_vector_bundles}). The fiber of the bundle over $p \in M$ will be denoted $E_p$. We define an \define{$n$-frame for $E$} to be a collection of sections $\sigma_i:M\to E$, $1\leq i\leq n$, so that for each $p\in M$ the collection $\{\sigma_i(p)\}_{i=1}^n$ spans $E_p$. There is an equivalent, but more conceptually convenient, definition of an $n$-frame for $E$ as a section of an associated fiber bundle, which we now describe.

\subsection{The frame and Parseval bundles.}\label{sec:parseval_bundle} Choosing $n$ sections of $E$ is equivalent to choosing \emph{a single} section of the vector bundle $\Hom(\RR^n,E)\coloneqq\Hom(M\times \RR^n,E)$---that is, the bundle whose fiber over $p \in M$ is the space of linear maps $\RR^n \to E_p$. In general, $\Hom(\RR^n,E)$ is a $(\GL(k)\times \GL(n))$-bundle; however, since we have assumed $E$ is orientable, we can reduce the structure group to $\SL(k)\times \SL(n)$, where $\SL(n)$ denotes the special linear group of operators on $\RR^n$. One advantage of this reduction is that the structure group will be connected and allows us to avoid certain technicalities arising from disconnected structure groups. The group $\SL(k)\times \SL(n)$ acts faithfully on $\F^n(\RR^k)$ (see Proposition \ref{prop:frames_homogeneous_space} and the surrounding discussion), so we can form the associated bundle to $\Hom(\RR^n,E)$ with fiber $\F^n(\RR^k)$.

\begin{definition}[Frame bundle]
    The bundle associated to  $\Hom(\RR^n,E)$ with fiber $\F^n(\RR^k)$ is called the \define{$n$-frame bundle of $E$}, and is denoted $\F^n(E)$. A section  $\sigma:M \to \F^n(E)$ is called an \define{$n$-frame on $E$}.
\end{definition}

\begin{remark}
    By the above discussion, a section of $\F^n(E)$ corresponds to a choice of $n$ sections of $E$ that pointwise span the fibers. Notice that $\F^n(E)$ is no longer a vector bundle, so there is no guarantee that sections exist. 
\end{remark}

\begin{remark}\label{rmk:k_frame_bundle}
    In classical differential topology, the \define{frame bundle} of an orientable rank-$k$ vector bundle $E \to M$ is the associated principal $\SL(k)$-bundle. The fiber of the frame bundle over $p \in M$ is the collection of bases for the vector space $E_p$. In our terminology, this is equivalent to the $k$-frame bundle $\F^k(E)$ (see Remark \ref{rmk:n_frame_definition}). 
\end{remark}

Next, assume that $E$ has been endowed with a choice of inner product on each fiber $E_p$ which varies smoothly in $p$. We will refer to such a vector bundle as a \define{Riemannian vector bundle} and the choice of fiberwise inner products as a \define{Riemannian structure on $E$}. A Riemannian structure on $E$ corresponds uniquely to a reduction of the structure group of $E$ from $\SL(k)$ to the special orthogonal group $\SO(k)$. The trivial bundle $M\times \RR^n$ admits a canonical inner product on fibers and so in this context we can reduce the structure group on $\Hom(\RR^n,E)$ to $\SO(k)\times \SO(n)$. Similar to the situation above, we have that $\SO(k)\times \SO(n)$ acts faithfully on $\P^n(\RR^k)$ (see Proposition \ref{prop:parseval_homogeneous}), so we can form the associated bundle to $\Hom(\RR^n,E)$ with fiber $\P^n(\RR^k)$. 

\begin{definition}[Parseval bundle]
    The bundle associated to $\Hom(\RR^n,E)$ with fiber $\P^n(\RR^k)$ is called the \define{$n$-Parseval bundle of $E$}, and is denoted $\P^n(E)$. A section  $\sigma:M \to \P^n(E)$ is called an \define{$n$-Parseval frame on $E$}.
\end{definition}

\begin{remark}
     A section of $\P^n(E)$ corresponds to a choice of $n$ sections of $E$ that pointwise form $n$-Parseval frames of the fibers of $E$, so that this definition agrees with the one we gave in the introduction. Again, there is no guarantee that sections of $\P^n(E)$ exist.
\end{remark} 

\begin{remark}\label{rmk:orthonormal_frame_bundle}
    Classically, the \define{orthonormal frame bundle} of an orientable rank-$k$ Riemannian vector bundle $E \to M$ is the associated principal $\SO(k)$-bundle; the fiber over $p \in M$ is the space or orthonormal bases for the inner product space $E_p$. In light of Remark \ref{rmk:parseval_orthonormal_basis}, this is equivalent to the $k$-Parseval bundle $\P^k(E)$. 
\end{remark}

\subsection{Topology and geometry of \texorpdfstring{$\P^n(V)$}{Parseval space}} Using the deformation retraction from Proposition \ref{prop:Parseval_hom_eq_to_frames} we can produce a similar deformation retraction between the Parseval and frame bundles.

 \begin{theorem}\label{thm:parseval_deformation_retract}
     There is a fiber-preserving strong deformation retract $\F^n(E) \to \P^n(E)$. 
 \end{theorem}

 \begin{proof}
     By the constructions described above, there is a natural inclusion $\overline{\iota}:\P^n(E)\to \F^n(E)$ given at the fiber level by the inclusion $\iota:\P^n(\RR^k)\to \F^n(\RR^k)$ that Proposition \ref{prop:Parseval_hom_eq_to_frames} shows is a homotopy equivalence. At the level of homotopy groups, we have the following commutative diagram whose rows come from the long exact sequence of the fibrations of the frame and Parseval bundles, respectively.

     $$
     \begin{tikzcd}
    \pi_{i+1}(M)\arrow[r]\arrow[d, "\mathsf{Id}",leftrightarrow] & \pi_i(\P^n(\RR^k))\arrow[r]\arrow[d,"\iota_\ast"]& \pi_i(\P^n(E))\arrow[r]\arrow[d,"\overline{\iota}_\ast"] &\pi_i(M)\arrow[r]\arrow[d,"\mathsf{Id}",leftrightarrow] &\pi_{i-1}(\P^n(\RR^k))\arrow[d,"\iota_\ast"]\\
    \pi_{i+1}(M)\arrow[r]& \pi_i(\F^n(\RR^k))\arrow[r]& \pi_i(\F^n(E))\arrow[r] &\pi_i(M)\arrow[r] &\pi_{i-1}(\F^n(\RR^k))
     \end{tikzcd}
     $$

     With the exception of the map $\overline{\iota}:\pi_i(\P^n(E))\to \pi_i(\F^n(E))$, we know that the remaining maps are all isomorphisms, so the five lemma implies that $\overline{\iota}$ is also an isomorphism. This implies that $\P^n(E)$ and $\F^n(E)$ are weakly homotopy equivalent, and by applying Whitehead's theorem we obtain a homotopy equivalence between $\P^n(E)$ and $\F^n(E)$ that commutes with the bundle projections and induces the above maps. By a standard result in algebraic topology (see \cite[Ch. 7, \S 5]{May_concise_alg_top}, for example) it then follows that the above homotopy equivalence is a fiber homotopy equivalence, and the result follows.
 \end{proof}

This theorem immediately implies the following corollary.

\begin{corollary}\label{cor:parseval_if_and_only_if_frame}
    A vector bundle $E \to M$ admits a Parseval frame of size $n$ (i.e., a section of $\P^n(E)$) if and only if it admits a frame of size $n$ (i.e., a section of $\F^n(E)$).
\end{corollary}

Finally, we give alternative, and somewhat more concrete, descriptions of the bundles $\F^n(E)$ and $\P^n(E)$. Let $\overline{G}=\SL(k)\times \SL(n)$ and recall that $\Hom(\RR^n,E)$ is a $\overline{G}$-bundle. We can thus form the associated principal $\overline{G}$-bundle $\overline{G}(E)$. The bundle $\overline{G}(E)$ admits a fiber-preserving right $\overline{G}$-action by virtue of being a principal bundle, so we get a \emph{left} action of $\overline{G}$ on $\overline{G}(E)\times \F^n(\RR^k)$ by $g\cdot (A,B)=(A\cdot g^{-1},g\cdot B)$. The quotient by this action is denoted by $\overline{G}(E)\times_{\overline{G}}\F^n(\RR^k)$. Analogously, if we let $G=\SO(k)\times \SO(n)$, then using the Riemannian structure on $E$ we can form the principal $G$-bundle $G(E)$ associated to $\Hom(\RR^n,E)$. The quotient of the action of $G$ on $G(E) \times \P^n(\RR^k)$ is denoted $G(E) \times_G \P^n(\RR^k)$. The following is a standard fact from bundle theory.

\begin{proposition}\label{prop:associated_fiber_bundles}
    The spaces $\overline{G}(E)\times_{\overline{G}} \F^n(\RR^k)$ and $\F^n(E)$ are isomorphic as $\F^n(\RR^k)$-bundles. Similarly, the spaces $G(E)\times_G \P^n(\RR^k)$ and $\P^n(E)$ are isomorphic as $\P^n(\RR^k)$-bundles.
\end{proposition}

\subsection{Complex vector bundles.}\label{sec:complex_vector_bundles} Here we explain how to extend these ideas to the setting of complex vector bundles. Let $E\to M$ be a rank-$k$ complex vector bundle. More precisely, with $\GL(\CC^k)\hookrightarrow \GL(\RR^{2k})$  the standard embedding, $E\to M$ is a rank-$2k$ real vector bundle along with a reduction of structure group from $\GL(\RR^{2k})$ to $\GL(\CC^k)$. Since $\SL(\CC^k)$ is a deformation retract of $\GL(\CC^k)$ there is no loss of generality in assuming that the structure group is $\SL(\CC^k)$. Suppose that $E$ is equipped with a smoothly varying family of positive-definite Hermitian products on each fiber. Such a vector bundle will be called \define{Hermitian} and the choice of such a family of fiberwise Hermitian forms will be called a \define{Hermitian structure} on $E$. The presence of a Hermitian structure on $E$ allows the structure group to be further reduced to $\SU(k)$. If we let $G_{\CC}=\SU(k)\times \SU(n)$ and define $\Hom(\CC^n,E)\coloneqq\Hom(M\times \CC^n,E)$, then $\Hom(\CC^n,E)$ is a vector bundle whose structure group can be reduced to $G_\CC$. As in the real case, we can define the $n$-Parseval bundle $\P^n(E)$ as the associated bundle $G_\CC(E)\times_{G_\CC} \P^n(\CC^k)$, where $G_\CC(E)$ is the $G_\CC$-principal bundle associated to $\Hom(\CC^n,E)$.

\subsection{Evidence for the optimality of Parseval frames}\label{subsec:numerical_experiment}

Extrapolating from the case of frames on individual vector spaces~\cite{munch1992noise,benedetto2001wavelet}, we expect that Parseval frames on vector bundles will be more robust to noise than general frames. While proving this is beyond the scope of the present paper, we did the following numerical experiment which seems to confirm this hypothesis in the case of frames of size $3$ on the unit sphere $S^2 \subset \RR^3$. Code for this experiment is available on Github~\cite{frames-on-manifolds-code}.

As our Parseval frame on $S^2$, we take the projections of the standard basis vectors to the tangent space at each point; that is, for each $p \in S^2$, $\sigma_i(p)$ is the orthogonal projection of the $i$th standard basis vector to $T_pS^2$ (this is guaranteed to produce a Parseval frame at each point---see Lemma \ref{lem:naimark_han_larson} below). Here and in what follows, for computational purposes we do not consider \textit{all} points on the sphere, but rather work at the points of Sloane's putatively optimal packing of 1592 points on the sphere~\cite{sloane}. So in practice we are representing our Parseval frame by 1592 triples of vectors, where the vectors in the $i$th triple lie in the tangent space to the $i$th point from Sloane's list.

We generated 1000 random frames on $S^2$ as follows. We first generated three smooth vector fields in $\RR^3$ using simplex noise~\cite{opensimplex} and, at each of the 1592 points on the sphere, normalized each of the three vectors at that point to be a unit vector (since these are not yet tangent to the sphere, this does not imply the existence of a nowhere-vanishing vector field on the sphere). Finally, then, to get a frame at that point we orthogonally project each of the three unit vectors to the tangent space at that point. We also generated 1000 smooth vector fields on $S^2$ by generating smooth vector fields on $\RR^3$ using simplex noise and projecting them to the tangent space at each point on the sphere.

For each of our 1001 frames (the Parseval frame and the 1000 random frames), we attempted to reconstruct each of the 1000 vector fields from noisy data. Specifically, at each $p \in S^2$ (again, $p$ is one of Sloane's 1592 points), the vector field determines $v(p) \in T_p S^2$, the frame gives $\sigma_1(p), \sigma_2(p), \sigma_3(p) \in T_p S^2$, and we computed the dot products of $v(p)$ with the $\sigma_i(p)$; the resulting triple of numbers is our data at the point. If $\sigma(p)$ is the $2 \times 3$ matrix whose columns are the coordinates of the $\sigma_i(p)$ with respect to some orthonormal basis for $T_pS^2$ and we write $v(p)$ as a 2-dimensional vector whose entries are the coordinates of $v(p)$ with respect to the same basis, then the data is $\sigma(p)^T v(p)$. 

We added Gaussian white noise with covariance matrix $0.01 \bI_{3}$ to the data $\sigma(p)^T v(p)$, then attempted to reconstruct the original vector by applying the reconstruction operator $(\sigma(p) \sigma(p)^T)^{-1} \sigma(p)$, which provides a perfect reconstruction in the absence of noise. Computing the squared error in this reconstruction gives a non-negative number at each of our 1592 points for each vector field, and the mean of these numbers gives a mean squared error (MSE) for the reconstruction of the vector field from the noise-corrupted data.

\begin{figure}
    \centering
    \includegraphics[width=2.3in]{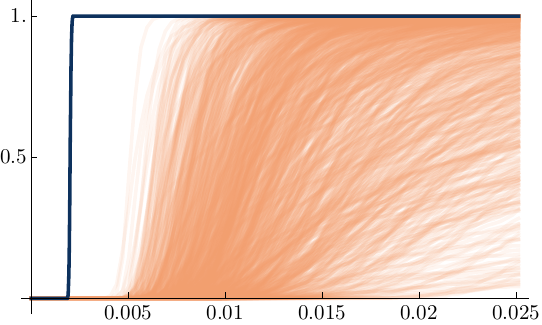} \qquad \includegraphics[width=2.3in]{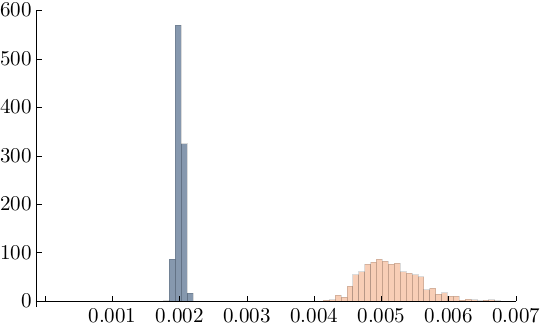}
    \caption{Left: plots of the empirical CDFs of the MSE distributions for the Parseval frame (dark blue) and the random frames (light orange). Right: histograms of the MSE distributions for the Parseval frame (dark blue) and for the best of the random frames (lighter orange).}
    \label{fig:numerical experiment}
\end{figure}

Therefore, each of our 1001 frames has an associated distribution of 1000 MSEs: one MSE for each of the 1000 random smooth vector fields. Overall, the Parseval frame has significantly smaller MSEs than the random frames: the MSEs of the Parseval frame all lie in the interval $(0.0018,0.0022)$, whereas the 1 million MSEs for the random frames (1000 MSEs for each of 1000 random frames) were in the interval $(0.0038,51.31)$. We visualize the MSE distributions in Figure~\ref{fig:numerical experiment}, which shows empirical cumulative distribution functions (CDFs) of the MSE distributions on the left and the histograms of MSE distributions for the Parseval frame and the best of the random frames on the right.

\section{Existence and Obstructions for Parseval frames}\label{sec:existence_of_frames}

In this section, we prove the existence of Parseval frames of sufficiently large size for general vector bundles. The proofs are very short and appeal to results from~\cite{Steenrod_Fibre_Bundles} in a rather blackbox fashion. The basic ideas rely on the \emph{obstruction theory} for sections of bundles. We then give improved existence results for tangent bundles for various families of manifolds, and provide a short discussion of the existence question from the perspective of classifying spaces.

\subsection{Existence of Parseval frames} Roughly speaking, the existence of sections of a bundle is obstructed by a sequence of cohomology classes with coefficients in the homotopy groups of the fiber. In our case, we know that the fiber of $\P^n(E)\to M$ is diffeomorphic to a Stiefel manifold (Proposition \ref{prop:Parseval_same_as_Stiefel}), and using the fact that such Stiefel manifolds have many trivial homotopy groups, we can get a bound on the number of vectors necessary to form a frame on $M$.

\begin{theorem}\label{thm:parseval_sections_exist}
Let $M$ be a compact $d$-dimensional manifold and let $E\to M$ be an orientable rank-$k$ Riemannian vector bundle. If $n\geq d+k$, then $\P^n(E)$ admits a section.
\end{theorem}

\begin{proof}
First, choose a triangulation for $M$---that is, a homeomorphism with a simplicial complex (triangulations of smooth manifolds always exist~\cite{cairns1935triangulation,whitehead1940c1}). The $n$-Parseval bundle is a bundle over $M$ whose fibers are diffeomorphic to $\P^n(\RR^k)\cong \St_k(\RR^n)$ and the homotopy group $\pi_i(\St_k(\RR^n))$ is trivial if $i< n-k$~\cite[\S 25.6]{Steenrod_Fibre_Bundles}. Given our hypothesis on $n$, this happens when $i<d$, in which case Corollary~29.3 of~\cite{Steenrod_Fibre_Bundles} implies that any section defined on the zero skeleton of $M$---that is, the inverse image of the vertex set of the simplicial complex under our triangulation---can be extended to the whole of $M$. Since the zero skeleton is a discrete set, it is trivial to choose a section on it, as this amounts to choosing a Parseval frame at each point independently. In general, the section constructed in this fashion will only be continuous, but an approximation theorem of Steenrod (see~\cite[\S 6.7]{Steenrod_Fibre_Bundles} or~\cite{Wockel_Steenrod_Approximation}) ensures that a continuous section can always be replaced by a smooth section.

\end{proof}

Next, we discuss the analogue of Theorem~\ref{thm:parseval_sections_exist} for complex vector bundles, which has a similar proof.

\begin{theorem}\label{thm:parseval_sections_exist_complex}
    Let $M$ be a compact $d$-dimensional manifold and let $E\to M$ be a rank-$k$ Hermitian vector bundle. If $n\geq \floor{d/2}+k$ then $\P^n(E)$ admits a section. 
\end{theorem}

\begin{proof}
    Triangulate $M$ as above. If $i\leq 2(n-k)$ then we have that the homotopy group of the fiber $\St_k(\CC^n)$ is trivial~\cite[\S 25.7]{Steenrod_Fibre_Bundles}. Our hypothesis on $n$ implies that $2n\geq d+2k$, so $i < 2(n-k)$ whenever $i<d$. Hence, we can apply Corollary~29.3 of~\cite{Steenrod_Fibre_Bundles} to show that any section defined on the 0-skeleton of $M$ can be extended to all of $M$. As in the proof of  Theorem~\ref{thm:parseval_sections_exist}, we can replace the \textit{a priori} continuous section with a smooth section.
\end{proof}


\subsection{Parseval frames for tangent bundles}\label{sec:applications_to_tangent bundles}
This subsection focuses on the special case where $M$ is a $d$-dimensional Riemannian manifold and the bundle under consideration is the tangent bundle $TM \to M$. The tangent bundle is rank-$d$, so Theorem~\ref{thm:parseval_sections_exist} implies it admits a Parseval frame of size $2d$. In fact, in the tangent bundle setting, an argument involving the Whitney immersion theorem (cf.\ \cite[\S 2]{freeman2012moving}) together with Theorem \ref{thm:parseval_deformation_retract} proves the existence of a frame of size $2d-1$. This size is, however, not optimal in general, and it is interesting to consider conditions on $M$ which guarantee the existence of a Parseval frame of size $n$, with $d \leq n < 2d-1$. In this range, the relevant homotopy groups of $\St_d(\RR^n)$ can be non-trivial, hence the proof strategy of Theorem \ref{thm:parseval_sections_exist} fails, and the existence of sections can be obstructed. By definition, a $d$-dimensional Riemannian manifold $M$ admits a Parseval frame of size $d$ on its tangent bundle (i.e., an orthonormal global frame) if and only if $M$ is parallelizable. In this section we will prove several theorems that provide conditions for a $d$-manifold to admit a Parseval frame of size $d+1$. The first result in this direction is a straightforward consequence of the fact that surfaces can be embedded in $\RR^3$.


\begin{proposition}\label{prop:parseval_surface}
Let $\Sigma$ be a closed, orientable surface endowed with a Riemannian structure. Then $T\Sigma$ admits a Parseval frame of size 3. 
\end{proposition}


Throughout this section, we make use of the following standard result from frame theory. This is a finite-dimensional version of the Naimark--Han--Larson Dilation Theorem.

\begin{lemma}[See, e.g., Theorem 2.2 of \cite{waldron2018introduction}]\label{lem:naimark_han_larson}
    Let $V$ be a $k$-dimensional real inner product space and let $W$ be an $n$-dimensional real inner product space, where $n \geq k$. The image of an orthonormal basis for $W$ under an orthogonal projection $W \to V$ defines an $n$-Parseval frame for $V$. Moreover, any $n$-Parseval frame for $V$ can be realized in this way. 
\end{lemma}

\begin{proof}[Proof of Proposition \ref{prop:parseval_surface}]
Embed $\Sigma$ in $\RR^3$ and endow it with the Riemannian metric induced by the embedding. At each point $p\in \Sigma$, we can take the orthogonal projection of the standard basis of $T_p \RR^3 \cong \RR^3$ to the tangent space $T_p\Sigma$. The resulting triple of vectors results in a Parseval frame of size $3$, by Lemma \ref{lem:naimark_han_larson}. Since Theorem~\ref{thm:parseval_deformation_retract} implies that the existence of a Parseval frame does not depend on the choice of Riemannian structure, this completes the proof.
\end{proof}

\begin{remark}
    We see from the proof of Proposition \ref{prop:parseval_surface} that the particular Riemannian structure on $\Sigma$ is not relevant to the existence of a Parseval frame, due to Theorem \ref{thm:parseval_deformation_retract}. This is a general observation, so we will sometimes suppress reference to a specific Riemannian structure below when discussing existence questions.
\end{remark}

Before proceeding with additional examples, we need to describe some additional structure possessed by $\P^{d+1}(TM)$ when $M$ is a smooth Riemannian $d$-manifold. The Parseval bundle $\P^{d+1}(TM)$ is a bundle over $M$ with fiber $\P^{d+1}(\RR^d)$. The space $\P^{d+1}(\RR^d)$ is diffeomorphic to the Stiefel manifold $\St_d(\RR^{d+1})$, which is diffeomorphic to $\SO(d+1)$. The structure group of the Parseval bundle is $G=\SO(d)\times \SO(d+1)$ and the $G$-action on the fiber $\P^{d+1}(\RR^d)\cong \SO(d+1)$ is given by $(A,B)\cdot C=\iota(A)CB^{-1}$, where $\iota:\SO(d)\to \SO(d+1)$ is the standard embedding. As such, this action factors through the natural action of $\tilde G\coloneqq\SO(d+1)\times \SO(d+1)$ on $\SO(d+1)$. Hence, we can regard the Parseval bundle as a $\tilde G$-bundle. 

Geometrically, one can interpret the above construction as follows. Start with $TM$ and replace it with its so-called \define{stable tangent bundle} $TM\oplus \RR$; here, we are abusing notation in that the second summand $\RR$ represents the trivial line bundle $M \times \RR \to M$. The Parseval bundle $\P^{d+1}(TM\oplus \RR)$ is naturally a $\tilde G$-bundle. Furthermore, the principal $\tilde G$-bundle associated to $\P^{d+1}(TM\oplus \RR)$ is the same as the principal $\tilde G$-bundle coming from the $\tilde G$-structure on $\P^{d+1}(TM\oplus \RR)$ described above. 

With this in mind we state the following definition (see \cite{Kervaire_Milnor_Homotopy_Spheres} for more details). 

\begin{definition}[Stably parallelizable]
    A $d$-manifold $M$ will be called \define{stably parallelizable} if $TM\oplus \RR$ is the trivial bundle.
\end{definition}

The following theorem is the crucial tool for constructing the examples of this section and shows the strong relationship between stable parallelizability and the existence of size $d+1$ Parseval frames.

\begin{theorem}\label{thm:stably_parallel}
If $M$ is a stably parallelizable orientable Riemannian $d$-manifold, then $TM$ admits a Parseval frame of size $d+1$. Moreover, the converse is also true when $H^1(M,\ZZtwo)=0$.
\end{theorem}

\begin{proof}
By our hypothesis, we have that $TM\oplus \RR$ is a trivial vector bundle. Define a Riemannian structure on $TM \oplus \RR$ by extending the structure from $TM$. Using triviality, we can choose a global orthonormal moving basis for $TM \oplus \RR$. By Lemma \ref{lem:naimark_han_larson}, the image of this basis under the orthogonal projection to the subbundle $TM \subset TM\oplus \RR$ defines a size-$(d+1)$ Parseval frame for $TM$.

Using the $\tilde G$-bundle perspective, we can alternatively deduce the forward direction from the stronger conclusion that the Parseval bundle $\P^{d+1}(TM)$ is trivial. By our hypothesis, we have that $TM\oplus \RR$ is a trivial vector bundle and hence $\Hom(\RR^{d+1},TM\oplus \RR)$ is also trivial. This latter bundle is a $\tilde G$-bundle and it follows that the corresponding principal $\tilde G$-bundle is trivial. Finally, since $\P^{d+1}(TM)$ is an associated bundle of this principal $\tilde G$-bundle, it follows that $\P^{d+1}(TM)$ is also trivial.


For the other direction, we use the generalization of the  Naimark--Han--Larson dilation theorem to Riemannian vector bundles proved in~\cite[Thm 1.1]{freeman2012moving}. By applying this result, we obtain a line bundle $E\to M$ over $M$ so that $TM\oplus E$ is a trivial bundle. It is well known that equivalence classes of line bundles over $M$ are in one-to-one correspondence with $H^1(M,\ZZtwo)$ (see \cite[Theorem 4.14]{cohenbundles}) and so our triviality hypothesis implies that $E$ is the trivial bundle and hence $M$ is stably parallelizable.
\end{proof}

We now describe several classes of $d$-manifolds which admit Parseval frames of size $d+1$. We already saw in Proposition~\ref{prop:parseval_surface} that closed, orientable surfaces form such a class. Next, recall Stiefel's theorem, which says that closed, orientable 3-manifolds have trivial tangent bundles~\cite{stiefel1935richtungsfelder}; that is, they admit Parseval frames of size $d=3$. Therefore, $4$-manifolds are the next interesting case to consider.



\begin{corollary}\label{cor:4_manifolds}
Let $M$ be a closed, orientable 4-manifold with vanishing second Stiefel--Whitney class and vanishing first Pontryagin class. Then $TM$ admits a Parseval frame of size 5. 
\end{corollary}

\begin{proof}
By a result of Cappell and Shaneson \cite{Cappell_Shaneson_Embeddings}, closed, orientable 4-manifolds whose second Stiefel--Whitney class and first Pontryagin classes both vanish are stably parallelizable, so the result then follows by applying Theorem \ref{thm:stably_parallel}.
\end{proof}

Let $M$ be a smooth $d$-manifold. Recall that $M$ is called a \define{homology sphere} if $H_i(M,\ZZ)\cong H_i(S^d,\ZZ)$ for all $i$. By work of Kervaire \cite[proof of Thm. 3]{Kervaire_smooth_homology_spheres} it follows that homology spheres are stably parallelizable. This gives the following immediate corollary:

\begin{corollary}\label{cor:homotopy_homology_spheres}
If $M$ is a $d$-dimensional homology sphere, then $M$ admits a Parseval frame of size $d+1$.
\end{corollary}

Another class of manifolds for which stable parallelizability is well-understood are \define{lens spaces}, whose definition we now recall. Let $p$ be a prime and $b=(b_0\ldots b_n)$ be a collection of integers such that $1\leq b_i\leq p-1$. Next, define a free $\ZZ/p\ZZ$ action on $\CC^{n+1}$ generated by the map 
$$(z_0,\ldots,z_n)\mapsto (\lambda^{b_0}z_0,\ldots, \lambda^{b_n}z_n),$$
where $\lambda$ is a fixed primitive $p$th root of unity. This action preserves the standard $(2n+1)$-sphere $S^{2n+1}\subset \CC^{n+1}$ and the \define{lens space} $L(p,b)$ is the $(2n+1)$-dimensional manifold obtained as the quotient of $S^{2n+1}$ by this action. The following result shows that deciding whether or not a lens space is stably parallelizable is a purely modular arithmetic problem.

\begin{proposition}[\!\!\cite{Ewing_etal_lens_space}]\label{prop:lens_space_stable_parallel}
    For odd $p$, the lens space $L(p,b)$ is stably parallelizable if and only if $n<p$ and
    $$b_0^{2j}+\cdots+b_n^{2j}=0 \pmod{p}$$
    for $1\leq j\leq \floor{n/2}$
\end{proposition}

By applying Deligne's solution to the Weil conjectures, the authors of \cite{Ewing_etal_lens_space} are able to show that there are many examples of stably parallelizable lens spaces. In particular they show that for each $n$ there are infinitely many stably parallelizable $(2n+1)$-dimensional lens spaces (see \cite[Cor.~2.2]{Ewing_etal_lens_space}). In our setting, this implies the following. 
\begin{corollary}
    For each positive integer $n$, if $p \equiv 1 \pmod{n+1}$, then there is a lens space $L(p,b)$ of dimension $2n+1$ which admits a Parseval frame of size $2n+2$.
\end{corollary}

\noindent These results were subsequently generalized by Kwak \cite{Kwak_lens_space}, who showed a similar condition for stable parallelizability of lens spaces obtained as quotients of smooth homotopy spheres.

In a complementary direction, we can also use Theorem \ref{thm:stably_parallel} to give many examples of $d$-manifolds that do not admit Parseval frames of size $d+1$. In particular:

\begin{theorem}\label{thm:no_d+1_frames}
Let $d$ be a natural number. Then there exists a closed orientable $d$-manifold that does not admit a Parseval frame of size $d+1$ if and only if $d>3$.
\end{theorem}

\begin{proof}
The only closed 1-manifold is $S^1$, which is parallelizable and hence admits Parseval frames of any size. By Proposition \ref{prop:parseval_surface} we have that all closed orientable surfaces admit a Parseval frame of size 3. We have also already seen that all closed orientable 3-manifolds are parallelizable (see \cite{stiefel1935richtungsfelder}) and hence admit Parseval frames of any size that is at least 3.

In the other direction, the proof breaks into two cases depending on whether $d$ is even or odd. In the odd case, write $d=2k+1$, and let $b$ be the $k+1$ tuple $(1,\ldots,1)$. If $k>2$, let $p=3$ and if $k=2$, let $p=5$ and consider the lens space $L(p,b)$. This lens space is obtained via a free action by an order $p$ group on $S^{d}$, so the quotient map $S^d\to L(p,b)$ is a universal covering and we find that $\pi_1(L(p,b)\cong \ZZ/p\ZZ$. Since $p$ is odd, this implies that $H^1(L(p,b),\ZZtwo)=0$. If $k>2$ then no value of $b$ allows $L(p,b)$ to be stably parallelizable and if $k=2$ then the tuple $b$ does not satisfy the appropriate modular equation, and it follows from Proposition~\ref{prop:lens_space_stable_parallel} that $L(p,b)$ is never stably parallelizable, and hence Theorem \ref{thm:stably_parallel} implies that no Parseval frame of size $d+1$ exists. 

When $d$ is even, consider $\CP^{d/2}$. It is well known that $\CP^{d/2}$ is a simply connected $d$-manifold and hence $H^1(\CP^{d/2},\ZZtwo)$ is trivial.  From the main theorem of \cite{TrewZvenGrassmannManifolds} we have that $\CP^n$ is stably parallelizable iff $n=1$. Since $d>3$ we have $d/2>1$ and so $\CP^{d/2}$ is not stably parallelizable, and hence by Theorem \ref{thm:stably_parallel} $\CP^{d/2}$ does not admit a Parseval frame of size $d+1$.

\end{proof}

\subsection{A classifying space perspective}\label{sec:obstructions}

We now provide a short discussion of the existence of Parseval frames from the perspective of classifying spaces. We begin with general constructions, so we temporarily let $G$ be an arbitrary group. There is a certain principal $G$-bundle, traditionally denoted as $EG\to BG$, that classifies all principal $G$-bundles over CW complexes. The base $BG$ of this bundle is known as the \define{classifying space} for $G$, $EG$ is called the \define{universal bundle} for $G$, and the sense in which this structure classifies principal $G$-bundles is as follows. Given a CW complex $Y$ and a continuous map $f:Y\to BG$, we can form a principal $G$-bundle $f^\ast(EG) \to Y$ via the pullback construction. In this construction, we call $f$ the \define{classifying map} of the principal $G$-bundle $f^\ast(EG)$. Some important properties of the universal bundle are summarized below.

\begin{proposition}[See, e.g., Sections 3.4 and 3.5 of \cite{rudolph2017differential}]
For a Lie group $G$ and CW complex $Y$:
\begin{itemize}
    \item the universal bundle $EG \to BG$ exists and is unique up to a natural notion of isomorphism;
    \item every principal $G$-bundle over $Y$ arises as a pullback $f^\ast(EG)$ for a classifying map $f:Y \to BG$;
    \item two $G$-bundles are isomorphic if and only if their classifying maps are homotopic;
    \item a principal $G$-bundle is trivial if the classifying map $f:Y \to BG$ can be lifted to $\tilde f:Y \to EG$.
\end{itemize}
\end{proposition}

We now specify to the setting of Parseval frames.

\begin{proposition}\label{prop:climbing_whithead_tower}
    Let $E\to M$ be an orientable rank-$k$ Riemannian vector bundle. Let $n>k$, let $G=\SO(k)\times \SO(n)$, and let $G(E)$ be the principal $G$-bundle associated to the Parseval bundle $\P^n(E)$ with corresponding classifying map $f:M\to BG$. If $f$ can be lifted to $\tilde f:M\to EG$, then $\P^n(E)$ is a trivial bundle. Moreover, $E$ admits a Parseval frame of size $n$. 
\end{proposition}

\begin{proof}
If the classifying map can be lifted to $\tilde f:M\to EG$ then we can define a section $\sigma:M\to G(E)\cong f^\ast(EG)$ via $\sigma(m)=(m,\tilde f(m))$. Since $G(E)$ is a principal bundle, such a section provides a global trivialization of $G(E)$. This implies that the structure group for $G(E)$ can be reduced to the trivial group. Since the transition functions for $\P^n(E)$ and $G(E)$ are the same, this implies $\P^n(E)$ is also trivial. The last assertion then follows trivially. 
\end{proof}

\section{Generic Sections and Spectral covers}\label{sec:spectra}

In this section, we consider the question of the existence of sections of generic frames.

\subsection{Frame spectra and genericity.}
Let $V$ be a $k$-dimensional real vector space equipped with an inner product. Let $A\in \F^n(V)$ be a frame and let $AA^T \in \End(V)$ be the corresponding frame operator, which is positive definite. If $AA^T$ has eigenvalues $\lambda_i$, we let $\spec(A) \coloneqq (\lambda_1,\ldots,\lambda_k)$, where we order the eigenvalues such that $\lambda_1 \geq \lambda_2 \geq \cdots \geq \lambda_k > 0$. This \define{spectral operator} can be considered as a map $\spec:\F^n(V) \to \Sc_k$, where $\Sc_k\coloneqq\{(x_1,\ldots, x_k) \in \RR^k \mid x_1\geq \ldots\geq x_k>0\}$. The interior of $\Sc_k$ is given by $\Sc_k^\circ\coloneqq\{(x_1,\ldots,x_k)\in \Sc_k \mid x_i>x_{i+1} \, \forall \, i\}$.

\begin{definition}
    A frame $A \in \F^n(V)$ is called \define{generic} if $\spec(A) \in \Sc_k^\circ$. 
\end{definition}

\begin{remark}\label{rmk:generic_terminology}
    The set $\spec^{-1}\big(\Sc_k^\circ\big)$ of generic frames is a dense open subset of $\F^n(V)$, so that the terminology is justified.
\end{remark}


For $x=(x_1,\ldots, x_k)\in \Sc_k$ we define the \define{multiplicity of the $i$th entry} as $
\mu_i(x)=\abs{\{j\mid x_i=x_j\}}$. By definition, a frame $A$ is generic if and only if $\mu_i(\spec(A)) = 1$ for all $i$. On the other hand, $A$ is tight if and only if $\mu_i(\spec(A)) = k$ for all $i$. In this sense, we can think of Parseval frames and generic frames as living at opposite ends of the gamut of genericity.

Next, suppose that $E\to M$ is an orientable rank-$k$ Riemannian vector bundle over a connected $d$-dimensional manifold. Theorem~\ref{thm:parseval_sections_exist} tells us that if $n \geq d + k$  then we will be able to find Parseval frames (i.e.\ sections of $\P^n(E)\to M$). With this context in mind, it is natural to ask the same question at the other extreme of genericity; that is, whether it is possible to find sections of $\F^n(M)$ that are everywhere generic. We now set precise notation and terminology regarding this question.

\begin{definition}
    Let $\sigma:M\to \F^n(E)$ be a section of the frame bundle. The associated \define{spectral operator} $\spec_\sigma:M\to \Sc_k$ is defined by $\spec_\sigma \coloneqq \spec \circ \sigma$. The associated \define{multiplicity functions} $\mu_{i,\sigma}:M\to \NN$ are defined by $\mu_{i,\sigma} \coloneqq \mu_i \circ \spec_\sigma$. We say that $\sigma$ is \define{generic} if $\spec_\sigma(M)\subset \Sc_k^\circ$ and that $\sigma$ is \define{$i$-generic} if $\mu_{i,\sigma}(p) = 1$ for all $p \in M$.
\end{definition}

\begin{remark}
    For any section $\sigma:M\to \F^n(E)$, $\spec_\sigma$ is continuous and $\mu_{i,\sigma}$ is upper semicontinuous.
\end{remark}

\subsection{An obstruction to generic sections.}
Naively, one might expect that most sections of $\F^n(E)$ will be generic. However, as we will see, the topology of both $M$ and $E$ can obstruct such global genericity of sections. To see this, we introduce the notion of \emph{spectral covers}. Let $\sigma:M \to \F^n(E)$ be a section and suppose that there is $1\leq i\leq k$ such that $\sigma$ is $i$-generic. The $i$-genericity condition means that the $i$th singular value of the frame $\sigma(p)$ has multiplicity 1 for each $p\in M$. This multiplicity-1 singular value gives rise to a 1-dimensional eigenspace in each fiber of $E$, or in other words a rank-$1$ subbundle of $E$. Using the Riemannian structure, we can pick out the two vectors of length 1 in each fiber. This gives rise to an $S^0$-bundle over $M$. Such an $S^0$-bundle over $M$ is nothing more than a (possibly disconnected) double cover of $M$. With this construction in mind we make the following definition. 

\begin{definition}
    Let $\sigma:M \to \F^n(E)$ be an $i$-generic section. The double cover of $M$ constructed from the induced rank-1 subbundle of $M$ is called the \define{$i$th spectral cover associated to $\sigma$}. 
\end{definition}

The next theorem shows how the topology of $M$ and $E$ can obstruct the existence of spectral covers and hence of generic sections. 

\begin{theorem}\label{thm:obstruction_to_generic_sections}
Let $E\to M$ be an oriented rank-$k$ Riemannian vector bundle over a connected $d$-manifold and let $\sigma:M \to \F^n(E)$ be a section. Suppose that $E\to M$ has non-trivial Euler class and that $H_1(M,\ZZtwo)=0$. Then $\sigma$ is not $i$-generic for any $1\leq i\leq k$.
\end{theorem}

\begin{proof}
Suppose for contradiction that there is some $i$ for which $\sigma$ is $i$-generic. Let $M_i\to M$ be the $i$th spectral cover associated to $\sigma$. First, we show that $M_i$ must be connected. Otherwise, $M_i$ is a disjoint union of two copies of $M$ and hence this cover admits a section. Since $M_i$ embeds in $E$ and is disjoint from the zero section it follows that we can construct a nowhere zero section $\tau:M\to E$. However, this contradicts our hypothesis that $E$ has non-trivial Euler class, hence $M_i$ is connected. 

By the above reasoning, we may assume that $M_i$ is a connected double cover. It is therefore regular and hence corresponds to the kernel of a non-trivial homomorphism $\phi:\pi_1(M)\to \ZZtwo$. However, $H_1(M,\ZZtwo)=0$ and so there is no such homomorphism and we arrive at a contradiction. 
\end{proof}

The following corollary shows how Theorem~\ref{thm:obstruction_to_generic_sections} can be applied to simply connected 4-manifolds

\begin{corollary}\label{cor:simply_connected_4_manifolds}
Let $M$ be a simply connected Riemannian 4-manifold. For any frame $\sigma\in \F^n(TM)$, $\sigma$ is not $i$-generic for any $1\leq i \leq 4$. 
\end{corollary}

\begin{proof}
Since $M$ is simply connected we have $H_1(M,\ZZtwo)=0$. Furthermore, if $M$ were non-orientable then it would have a connected orientation double cover. This would correspond to a non-trivial element of $H_1(M,\ZZtwo)$, and thus $M$ and hence $TM$ are both orientable. 

For a tangent bundle, vanishing of the Euler class is the same as having zero Euler characteristic. Let $b_i=\dim(H_i(M,\QQ))$ be the $i$th Betti number of $M$. The Euler characteristic can be computed as $\chi(M)=\sum_{i=0}^4(-1)^ib_i$. Since $M$ is simply connected, we have $b_1=0$ and by Poincar\'e duality it follows that $b_3=0$. Since $M$ is connected, we have $b_0=1$, and therefore $\chi(M)>0$. The result then follows by Theorem \ref{thm:obstruction_to_generic_sections}.
\end{proof}

\begin{remark}\label{rem:lots_of_4_manifolds}
    There are infinitely many simply connected 4-manifolds, including $S^4$, $S^2\times S^2$, $\CP^2$, rational surfaces, elliptic surfaces, and various fiber sums, surgeries, blowups, and blowdowns thereon, to which Corollary~\ref{cor:simply_connected_4_manifolds} can be applied (see~\cite{Stern_4_manifolds} and the references therein for details).
\end{remark}

\bibliographystyle{plain}
\bibliography{bibliography}

\begin{thebibliography}{10}

\bibitem{agrawal2015fiber}
Devanshu Agrawal and Jeff Knisley.
\newblock Fiber bundles and {Parseval} frames.
\newblock Preprint, \href{https://arxiv.org/abs/1512.03989}{\tt
  arXiv:1512.03989 [math.FA]}, 2015.

\bibitem{auckly2018parseval}
David Auckly and Peter Kuchment.
\newblock On {Parseval} frames of exponentially decaying composite {Wannier}
  functions.
\newblock {\em Mathematical Problems in Quantum Physics}, 717:227--240, 2018.

\bibitem{benedetto2001wavelet}
John~J. Benedetto and Oliver~M. Treiber.
\newblock Wavelet frames: {Multiresolution} analysis and extension principles.
\newblock In Lokenath Debnath, editor, {\em Wavelet Transforms and
  Time-Frequency Signal Analysis}, Applied and Numerical Harmonic Analysis,
  pages 3--36. Birkhäuser, Boston, MA, USA, 2001.

\bibitem{bodmann2005frames}
Bernhard~G. Bodmann and Vern~I. Paulsen.
\newblock Frames, graphs and erasures.
\newblock {\em Linear Algebra and its Applications}, 404:118--146, 2005.

\bibitem{bownik2022parseval}
Marcin Bownik, Karol Dziedziul, and Anna Kamont.
\newblock Parseval wavelet frames on {Riemannian} manifold.
\newblock {\em The Journal of Geometric Analysis}, 32:1--43, 2022.

\bibitem{bronstein2021geometric}
Michael~M. Bronstein, Joan Bruna, Taco Cohen, and Petar Veli{\v{c}}kovi{\'c}.
\newblock {Geometric Deep Learning: Grids, Groups, Graphs, Geodesics, and
  Gauges}.
\newblock Preprint, \href{https://arxiv.org/abs/2104.13478}{\tt
  arXiv:2104.13478 [cs.LG]}, 2021.

\bibitem{cahill2017connectivity}
Jameson Cahill, Dustin~G. Mixon, and Nate Strawn.
\newblock Connectivity and irreducibility of algebraic varieties of finite unit
  norm tight frames.
\newblock {\em SIAM Journal on Applied Algebra and Geometry}, 1(1):38--72,
  2017.

\bibitem{cairns1935triangulation}
Stewart~S. Cairns.
\newblock Triangulation of the manifold of class one.
\newblock {\em Bulletin of the American Mathematical Society}, 41(8):549--552,
  1935.

\bibitem{Cappell_Shaneson_Embeddings}
Sylvanin~E. Cappell and Julius~L. Shaneson.
\newblock Embeddings and immersions of four-dimensional manifolds in {${\bf
  R}\sp{6}$}.
\newblock In {\em Geometric Topology ({P}roceedings of the {G}eorgia {T}opology
  {C}onference, {A}thens, {GA}, 1977)}, pages 301--303. Academic Press, New
  York--London, 1979.

\bibitem{casazza2003equal}
Peter~G. Casazza and Jelena Kova{\v{c}}evi{\'c}.
\newblock Equal-norm tight frames with erasures.
\newblock {\em Advances in Computational Mathematics}, 18:387--430, 2003.

\bibitem{cohenbundles}
Ralph~L. Cohen.
\newblock {Bundles, Homotopy, and Manifolds}.
\newblock \url{http://math.stanford.edu/~ralph/book.pdf}, 2023.

\bibitem{dykema2003manifold}
Ken Dykema and Nate Strawn.
\newblock Manifold structure of spaces of spherical tight frames.
\newblock {\em International Journal of Pure and Applied Mathematics},
  28(2):217--256, 2006.

\bibitem{Ewing_etal_lens_space}
John Ewing, Suresh Moolgavkar, Larry Smith, and Robert~E. Stong.
\newblock Stable parallelizability of lens spaces.
\newblock {\em Journal of Pure and Applied Algebra}, 10(2):177--191, 1977/78.

\bibitem{freeman2012moving2}
Daniel Freeman, Ryan Hotovy, and Eileen Martin.
\newblock Moving finite unit norm tight frames for {$S^n$}.
\newblock {\em Illinois Journal of Mathematics}, 58(2):311--322, 2014.

\bibitem{freeman2012moving}
Daniel Freeman, Daniel Poore, Ann~Rebecca Wei, and Madeline Wyse.
\newblock Moving {P}arseval frames for vector bundles.
\newblock {\em Houston Journal of Mathematics}, 40(3):817--832, 2014.

\bibitem{han2000frames}
Deguang Han and David~R. Larson.
\newblock Frames, bases and group representations.
\newblock {\em Memoirs of the American Mathematical Society}, 147(697), 2000.

\bibitem{Kervaire_smooth_homology_spheres}
Michel~A. Kervaire.
\newblock Smooth homology spheres and their fundamental groups.
\newblock {\em Transactions of the American Mathematical Society}, 144:67--72,
  1969.

\bibitem{Kervaire_Milnor_Homotopy_Spheres}
Michel~A. Kervaire and John~W. Milnor.
\newblock Groups of homotopy spheres. {I}.
\newblock {\em Annals of Mathematics. Second Series}, 77:504--537, 1963.

\bibitem{kovavcevic2008introduction}
Jelena Kova{\v{c}}evi{\'c} and Amina Chebira.
\newblock An introduction to frames.
\newblock {\em Foundations and Trends in Signal Processing}, 2(1):1--94, 2008.

\bibitem{kuchment2008tight}
Peter Kuchment.
\newblock Tight frames of exponentially decaying {Wannier} functions.
\newblock {\em Journal of Physics A: Mathematical and Theoretical},
  42(2):025203, 2008.

\bibitem{kuchment2016overview}
Peter Kuchment.
\newblock An overview of periodic elliptic operators.
\newblock {\em Bulletin of the American Mathematical Society}, 53(3):343--414,
  2016.

\bibitem{Kwak_lens_space}
Jin~Ho Kwak.
\newblock The stable parallelizability of a smooth homotopy lens space.
\newblock {\em Journal of Pure and Applied Algebra}, 50(2):155--169, 1988.

\bibitem{May_concise_alg_top}
J.~Peter May.
\newblock {\em A Concise Course in Algebraic Topology}.
\newblock Chicago Lectures in Mathematics. University of Chicago Press,
  Chicago, IL, USA, 1999.

\bibitem{melzi2019gframes}
Simone Melzi, Riccardo Spezialetti, Federico Tombari, Michael~M. Bronstein,
  Luigi~Di Stefano, and Emanuele Rodola.
\newblock Gframes: Gradient-based local reference frame for {3D} shape
  matching.
\newblock In {\em Proceedings of the IEEE/CVF Conference on Computer Vision and
  Pattern Recognition}, pages 4629--4638, 2019.

\bibitem{mixon2021three}
Dustin~G. Mixon, Tom Needham, Clayton Shonkwiler, and Soledad Villar.
\newblock Three proofs of the {Benedetto--Fickus} theorem.
\newblock Preprint, \href{https://arxiv.org/abs/2112.02916}{\tt
  arXiv:2112.02916 [math.MG]}, 2021.

\bibitem{munch1992noise}
Niels~Juul Munch.
\newblock Noise reduction in tight {Weyl-Heisenberg} frames.
\newblock {\em IEEE Transactions on Information Theory}, 38(2):608--616, 1992.

\bibitem{needham2021symplectic}
Tom Needham and Clayton Shonkwiler.
\newblock Symplectic geometry and connectivity of spaces of frames.
\newblock {\em Advances in Computational Mathematics}, 47(1):5, 2021.

\bibitem{needham2022admissibility}
Tom Needham and Clayton Shonkwiler.
\newblock Admissibility and frame homotopy for quaternionic frames.
\newblock {\em Linear Algebra and its Applications}, 645:237--255, 2022.

\bibitem{needham2022toric}
Tom Needham and Clayton Shonkwiler.
\newblock Toric symplectic geometry and full spark frames.
\newblock {\em Applied and Computational Harmonic Analysis}, 61:254--287, 2022.

\bibitem{needham2023fusion}
Tom Needham and Clayton Shonkwiler.
\newblock Fusion frame homotopy and tightening fusion frames by gradient
  descent.
\newblock {\em Journal of Fourier Analysis and Applications}, 29(4):51, 2023.

\bibitem{rudolph2017differential}
Gerd Rudolph and Matthias Schmidt.
\newblock {\em Differential Geometry and Mathematical Physics: Part II. Fibre
  Bundles, Topology and Gauge Fields}.
\newblock Springer, Dordrecht, 2017.

\bibitem{frames-on-manifolds-code}
Clayton Shonkwiler.
\newblock {Frames on Manifolds}.
\newblock \url{https://github.com/shonkwiler/frames-on-manifolds}.

\bibitem{sloane}
Neil J.~A. Sloane, with the collaboration of Ronald~H. Hardin, Warren~D. Smith,
  et~al.
\newblock Tables of spherical codes.
\newblock \url{http://neilsloane.com/packings/}.

\bibitem{opensimplex}
Kurt Spencer.
\newblock {OpenSimplex} 2.
\newblock \url{https://github.com/KdotJPG/OpenSimplex2}.

\bibitem{Steenrod_Fibre_Bundles}
Norman Steenrod.
\newblock {\em The {T}opology of {F}ibre {B}undles}.
\newblock Princeton Mathematical Series, vol. 14. Princeton University Press,
  Princeton, NJ, USA, 1951.

\bibitem{Stern_4_manifolds}
Ronald~J. Stern.
\newblock Will we ever classify simply-connected smooth 4-manifolds?
\newblock In David~A. Ellwood, Peter~S. Ozsv\'ath, Andr\'as~I. Stipsicz, and
  Zolt\'an Szab\'o, editors, {\em Floer Homology, Gauge Theory, and
  Low-Dimensional Topology}, volume~5 of {\em Clay Mathematics Proceedings},
  pages 225--239. American Mathematical Society, Providence, RI, USA, 2006.

\bibitem{stiefel1935richtungsfelder}
Eduard Stiefel.
\newblock {Richtungsfelder und Fernparallelismus in $n$-dimensionalen
  Mannigfaltigkeiten}.
\newblock {\em Commentarii Mathematici Helvetici}, 8(1):305--353, 1935.

\bibitem{strawn2011finite}
Nate Strawn.
\newblock Finite frame varieties: {N}onsingular points, tangent spaces, and
  explicit local parameterizations.
\newblock {\em Journal of Fourier Analysis and Applications}, 17(5):821--853,
  2011.

\bibitem{TrewZvenGrassmannManifolds}
S.~Trew and P.~Zvengrowski.
\newblock Nonparallelizability of {G}rassmann manifolds.
\newblock {\em Canad. Math. Bull.}, 27(1):127--128, 1984.

\bibitem{waldron2018introduction}
Shayne F.~D. Waldron.
\newblock {\em An Introduction to Finite Tight Frames}.
\newblock Applied and Numerical Harmonic Analysis. Birkhäuser, New York, NY,
  USA, 2018.

\bibitem{whitehead1940c1}
J.~H.~C. Whitehead.
\newblock On {$C^1$}-complexes.
\newblock {\em Annals of Mathematics. Second Series}, 41(4):809--824, 1940.

\bibitem{Wockel_Steenrod_Approximation}
Christoph Wockel.
\newblock A generalization of {S}teenrod's approximation theorem.
\newblock {\em Archivum Mathematicum (Brno)}, 45(2):95--104, 2009.

\end{thebibliography}

\end{document}